\documentclass[11pt,a4paper]{article}
\usepackage{epsf,epsfig,amsfonts,amsgen,amsmath,amstext,amsbsy,amsopn,amsthm,lineno}
\usepackage[dvips]{color}
\usepackage{graphicx}

\newtheorem{theorem}{Theorem}

\newtheorem{lemma}{Lemma}
\newtheorem{cor}{Corollary}

\theoremstyle{definition}

\newtheorem{claim}{Claim}

\newtheorem{case}{Case}

\begin{document}
\title
{\Large\bf Extreme tenacity of graphs with given order and
size\thanks{Supported by NSFC (No.~60642002, 10871158 and
10861009)}}

\author{
T.C.E. Cheng$^a$, Yinkui Li$^b$, Chuandong Xu$^c$ and Shenggui
Zhang$^{c,d}$\thanks{Corresponding author. E-mail address:
sgzhang@nwpu.edu.cn (S. Zhang).}\\
\small $^a$Department of Logistics and Maritime Studies, The Hong Kong Polytechnic University,\\
\small Hung Hom, Kowloon, Hong Kong\\[2mm]
\small $^b$Department of Mathematics, Qinghai Nationalities College,\\
\small Xining, Qinghai 810000, P.R.~China\\[2mm]
\small $^c$Department of Applied Mathematics, Northwestern
Polytechnical University,\\
\small Xi'an, Shaanxi 710072, P.R.~China\\[2mm]
\small $^d$State Key Laboratory for Manufacturing Systems
Engineering, Xi'an Jiaotong University,\\
\small Xi'an, Shaanxi, 710049, P.R.~China}
\date{June, 5, 2010}
\maketitle

\begin{abstract}
Computer or communication networks are so designed that they do not
easily get disrupted under external attack and, moreover, these are
easily reconstructible if they do get disrupted. These desirable
properties of networks can be measured by various graph parameters,
such as connectivity, toughness, scattering number, integrity,
tenacity, rupture degree and edge-analogues of some of them. Among
these parameters, the tenacity and rupture degree are two better
ones to measure the stability of a network. In this paper we
consider two extremal problems on the tenacity of graphs: Determine
the minimum and maximum tenacity of graphs with given order and
size. We give a complete solution to the first problem, while for
the second one, it turns out that the problem is much more
complicated than that of the minimum case. We determine the maximum
tenacity of trees and unicyclic graphs with given order and show the
corresponding extremal graphs. These results are helpful in
constructing stable networks with lower costs. The paper concludes
with a discussion of a related problem on the edge vulnerability
parameters of graphs.

\medskip
\noindent {\bf Keywords:}  vulnerability parameters; tenacity;
extreme values; trees; unicyclic graphs
\end{abstract}

\section{Introduction}

In an analysis of the vulnerability of networks to disruption, three
quantities (there may be others) that come to the mind are: (1) the
number of elements that are not functioning; (2) the number of
remaining connected subnetworks and (3) the size of a largest
remaining group within which mutual communication can still occur.
Based on these quantities, many graph parameters, such as
connectivity, toughness \cite{Chvatal}, scattering number
\cite{Jung}, integrity \cite{Barefoot_Entringer_Swart_0}, tenacity
\cite{Cozzen_Moazzami_Stueckle_0}, rupture degree \cite{Li_Zhang_Li}
and edge-analogues of some of them have been proposed for measuring
the vulnerability of networks.

Throughout the paper, we use Bondy and Murty \cite{Bondy_Murty} for
terminology and notation not defined here. For a graph $G$, by
$\omega(G)$ we denote the number of components of $G$, and $\tau(G)$
the order of a largest component of $G$. We shall use $\lfloor
x\rfloor$ for the largest integer not larger than $x$ and $\lceil
x\rceil$ the smallest integer not smaller than $x$.

The connectivity is a parameter defined based on Quantity (1). The
{\it connectivity} of a noncomplete graph $G$ is defined by
\begin{align*}
    \kappa (G)=\min\{|X|:X\subset V(G), \omega (G-X)>1\},
\end{align*}
and that of the complete graph $K_n$ is defined as $n-1$.

Both toughness and scattering number take into account Quantities
(1) and (2). The {\it toughness} and {\it scattering number} of a
noncomplete connected graph $G$ are defined by
\begin{align*}
t(G)=\min\{\frac{|X|}{\omega(G-X)}:X\subset V(G), \omega(G-X)>1\}
\end{align*}
and
\begin{align*}
s(G)=max\{\omega(G-X)-|X|:X\subset V(G),\omega(G-X)>1\},
\end{align*}
respectively.

The integrity  is defined based on Quantities (1) and (3). The {\it
integrity} of a graph $G$ is defined by
\begin{align*}
I(G)=\min\{|X|+\tau(G-X):X\subset V(G)\}.
\end{align*}

Both the tenacity and rupture degree  take into account all the
three quantities. The {\it tenacity} and {\it rupture degree} of a
noncomplete connected graph $G$ are defined by
\begin{align*}
T(G)=\min\{\frac{|X|+\tau(G-X)}{\omega(G-X)}:X\subset
V(G),\omega(G-X)>1\}
\end{align*}
and
$$
r(G)=\max\{\omega(G-X)-|X|-\tau(G-X)\}: X\subset V(G),
\omega(G-X)>1\},
$$
respectively.

From the above definitions, we can see that the connectivity of a
graph reflects the difficulty in breaking down a network into
several pieces. This invariant is often too weak, since it does not
take into account what remains after the corresponding graph is
disconnected. Unlike the connectivity, each of the other
vulnerability measures, i.e., toughness, scattering number,
integrity, tenacity and rupture degree, reflects not only the
difficulty in breaking down the network but also the damage that has
been caused. Further, we can easily see that the tenacity and
rupture degree are the two most advanced ones among these parameters
when measuring the stability of networks.

When designing stable networks, it is often required to know the
structure of networks attaining the maximum and minimum values of a
given stability parameters with prescribed number of communications
stations and links. This problem was first studied by the well-known
graph theorist Frank Harary \cite{Harary}. Harary states that among
all the graphs with $n$ vertices and $m$ edges, the maximum
connectivity is $0$ when $m<n-1$ and $\lfloor\frac{2m}{n}\rfloor$
when $m\geq n-1$. For two integers $n$ and $m$ with $n\leq m\leq
{n\choose 2}$, Harary constructed graphs with $n$ vertices, $m$
edges and connectivity $\lfloor\frac{2m}{n}\rfloor$, which are now
widely known as the Harary graphs. Harary \cite{Harary} also
considered the minimum connectivity of graphs with a given number of
vertices and edges. He showed that among all the graphs with $n$
vertices and $m$ edges, the minimum connectivity is $0$ or
$m-{n-1\choose 2}$, whichever is larger. This lower bound on the
connectivity can be achieved by any graph consisting of a complete
subgraph $K_{n-1}$, together with exactly one additional vertex that
is adjacent to any $m-{n-1\choose 2}$ vertices of $K_{n-1}$.

As connectivity, it is natural to ask what are the extreme values
for each of these new vulnerability parameters of a graph with a
given number of vertices and edges. This problem has been studied in
the literature for toughness, scattering number and integrity. We
list the results on the extreme values of these vulnerability
parameters of graphs in the following table:

\begin{center}
{Table 1: Achievements of the study for extreme values of vulnerability parameters}\\

\medskip
\begin{tabular}{|l|l|l|}
  \hline
  Vulnerability parameters & Maximum value & Minimum value \\
  \hline
  (Edge-)Connectivity & Complete solution \cite{Harary}   & Complete solution \cite{Harary} \\
  \hline
  Toughness & Partial solution  \cite{Doty,Doty_Ferland_0,Doty_Ferland_5,Ferland,Goddard_Swart} & Unknown \\
  \hline
  Scattering number & Complete solution  \cite{Ouyang} &  Complete solution  \cite{Ouyang}\\
  \hline
  Integrity & Partial solution \cite{Barefoot_Entringer_Swart_5} & Complete solution  \cite{Ma_Liu} \\
  \hline
\end{tabular}
\end{center}

\medskip
In this paper we consider the problem of determining the extreme
tenacity of a graph with a given number of vertices and edges. We
give a complete solution to the problem for the minimum case in
Section 2. The problem for the maximum case is much more
complicated. In Section 3 we give a partial solution to this problem
by determining the maximum tenacity of trees and unicyclic graphs
with given number of vertices and show the corresponding extremal
graphs. We conclude the paper with a discussion of a related problem
on the edge-vulnerability parameters of graphs in the final section.

\section{Minimum tenacity of graphs}

\begin{theorem}
Let $n$ and $m$ be two positive integers with $n-1\leq m\leq
{n\choose 2}-1$. Then among all the connected graphs with $n$
vertices and $m$ edges, the minimum tenacity is $\frac{k+1}{n-k}$,
where ${k \choose 2}+(n-k)(k-1)< m\leq{k \choose 2}+(n-k)k$.
\end{theorem}

\begin{proof}
{Suppose that $G$ is a connected graph with $n$ vertices and $m$
edges such that its tenacity is minimum. Let $X^*$ be a vertex cut
of $G$ with
$$
T(G)=\frac{|X^*|+\tau(G-X^*)}{\omega(G-X^*)}.
$$
We assume that $X^*$ is chosen such that $|X^*|$ is as large as
possible. Denote the components of $G-X^*$ with at least two
vertices by $G_1,G_2,\ldots,G_p$.}

{Suppose $p\geq 2$.} Choose a vertex $u_{i}$ in $G_{i}$  such that
$u_i$ is adjacent to at least one vertex of $X^*$ for each $i$ with
$1\leq i\leq p$. Replace each edge $u_iv$ in $G_i$ by a new edge
$u_pv$ for every $i$ with $1\leq i\leq p-1$. Denote the resulting
graph by $G'$. Then $G'$ is also a connected graph with $n$ vertices
and $m$ edges, and $X^{'}=X^{\ast}\cup\{u_{p}\}$ is a vertex cut of
$G^{'}$ with
$$
\omega(G^{'}-X^{'})\geq \omega(G-X^{\ast})+(p-1)\geq
\omega(G-X^{\ast})+1
$$
and
$$
\tau(G^{'}-X^{'})\leq \tau(G-X^{\ast})-1.
$$
This implies that
\begin{align*}
T(G^{'})&\leq \frac{|X^{'}|+\tau(G^{'}-X^{'})}{\omega(G^{'}-X^{'})}
        \leq\frac{|X^{\ast}|+1+\tau(G-X^{\ast})-1}{\omega(G-X^{\ast})+1}\\
        &<
        \frac{|X^{\ast}|+\tau(G-X^{\ast})}{\omega(G-X^{\ast})}=T(G),
\end{align*}
{a contradiction.}

{Suppose now $p=1$.} We distinguish two cases.

\begin{case}
{$V(G_{1})$ is a clique.}
\end{case}

Let $u_1$ be a vertex in $G_1$. Set $X^{**}=X^{\ast}\cup
(V(G_{1})\setminus\{u_{1}\})$. Clearly we have
$$
\mbox{$\tau(G-X^{**})=1$ and $\omega(G-X^{**})=\omega(G-X^{\ast})$.}
$$
Then $X^{**}$ is a vertex cut of $G$ with
$$
\mbox{$|X^{**}|=|X^{\ast}|+\tau(G-X^{\ast})-1$}
$$
and
$$
\frac{|X^{**}|+\tau(G-X^{**})}{\omega(G-X^{**})}
=\frac{|X^{\ast}|+\tau(G-X^{\ast})}{\omega(G-X^{\ast})}=T(G),
$$
{contradicting the choice of $X^{\ast}$.}

\begin{case}
{$V(G_{1})$ is a not a clique.}
\end{case}

In this case, let $X_1$ be a vertex cut of $G_{1}$. Set
$X^{**}=X^{\ast}\cup X_{1}$. Then $X^{**}$ is a vertex cut of $G$
with
$$
\mbox{$|X^{**}|=|X^{\ast}|+|X_{1}|$, $\tau(G-X^{**})\leq
\tau(G-X^{\ast})-|X_{1}|-1$},
$$
and
$$
\omega(G-X^{**})\geq \omega(G-X^{\ast})+1.
$$
This implies that
\begin{align*}
\frac{|X^{**}|+\tau(G-X^{**})}{\omega(G-X^{**})}
        &\leq\frac{|X^{\ast}|+|X_{1}|+\tau(G-X^{\ast})-|X_{1}|-1}{\omega(G-X^{\ast})+1}\\
        &<
        \frac{|X^{\ast}|+\tau(G-X^{\ast})}{\omega(G-X^{\ast})}=T(G),
\end{align*}
{contradicting the definition of the tenacity of $G$.}

{From the above discussion, we have $p=0$, i.e., $\tau(G-X^*)=1$.}

Now let $|X^*|=x$. Then
\begin{align*}
T(G)=\frac{|X^{\ast}|+\tau(G-X^{\ast})}{\omega(G-X^{\ast})}=\frac{x+1}{n-x}.
\end{align*}

We claim that $x\geq k$. Otherwise,
\begin{align*}
 m&=|E(G)|\\
  &\leq{x\choose 2}+(n-x)x\\
        &\leq{k-1\choose 2}+(n-k+1)(k-1)\\
        &={k\choose 2}+(n-k)(k-1),
\end{align*}
a contradiction.

Therefore, we have
$$
T(G)=\frac{x+1}{n-x}\geq \frac{k+1}{n-k}.
$$
On the other hand, since ${k \choose 2}+(n-k)(k-1)< m\leq{k \choose
2}+(n-k)k$, it is easy to construct a connected graph with $n$
vertices and $m$ edges such that its tenacity is
$$
\frac{k+1}{n-k},
$$
which completes the proof of the theorem.
\end{proof}

From the proof of Theorem 1, we deduce the following result:

\begin{cor}
Let $n$ and $m$ be two positive integers with $n-1\leq m\leq
{n\choose 2}-1$, and $G$ be a connected graph with $n$ vertices and
$m$ edges such that its tenacity is minimum. Then $G$ consists of a
complete subgraph $K_k$ together with $n-k$ additional vertices
incident to $m-{k\choose 2}$ edges with the other end vertices in
$K_k$, where ${k \choose 2}+(n-k)(k-1)< m\leq{k \choose 2}+(n-k)k$.
\end{cor}

\section{Maximum tenacity of graphs}

In this section we consider the maximum tenacity of connected graphs
with given number of vertices and edges. It turns out that this
problem is much more complicated than that of the minimum case. Here
we give the results for the problem involving trees and unicyclic
graphs.

In the following by an odd (or even) path we mean a path with an odd
(or even) number of vertices. For a unicyclic graph $G$, we use
$C_G$ to denote the unique cycle in $G$ and by $U_G$ to denote the
set of vertices on $C_G$ with degree at least 3.

We first list some lemmas.

\begin{lemma}[Cozzen, Moazzami and Stueckle \cite{Cozzen_Moazzami_Stueckle_0}]
Let $G$ be a noncomplete connected graph and $H$ be a connected
spanning subgraph of $G$. Then $T(H)\leq T(G)$.
\end{lemma}

\begin{lemma}
Let $G$ be a connected graph. If there exists a vertex cut $X_0$ of
$G$ such that $G-X_0$ is a forest and $\omega(G-X_0)\geq |X_0|+2$
(resp. $\omega(G-X_0)\geq |X_0|+3$), then $T(G)\leq 1$ (resp. $T(G)<
1$).
\end{lemma}

\begin{proof}
If the maximum degree of $G-X_0$ is at most 1, then let $X^*=X_0$.
Otherwise, choose a vertex $v_1\in V(G-X_0)$ with
$d_{G-X_0}(v_1)\geq 2$ and set $X_1=X_0\cup \{v_1\}$. Then we have
$\omega(G-X_1)\geq |X_1|+2$. Repeating this process, we can finally
obtain a vertex cut $X_k$ of $G$ with $\omega(G-X_k)\geq |X_k|+2$
and the maximum degree of $G-X_k$ is at most 1. Choose $X_k$ as
$X^*$. Then we have
$$
\frac{|X^*|+\tau(G-X^*)}{\omega(G-X^*)}\leq
\frac{\omega(G-X^*)-2+2}{\omega(G-X^*)}=1,
$$
which implies that $T(G)\leq 1$ by the definition of tenacity. The
other assertion can be proved similarly.
\end{proof}

\begin{lemma}[Choudum and Priya \cite{Choudum_Priya}]
The tenacity of the path $P_n$ is
$$
T(P_n)=\left\{
         \begin{array}{ll}
           1, & \hbox{if $n$ is odd;} \\
           \frac{n+2}{n}, & \hbox{if $n$ is even.}
         \end{array}
       \right.
$$
\end{lemma}

\begin{lemma}[Cozzen, Moazzami and Stueckle \cite{Cozzen_Moazzami_Stueckle_5}]
The tenacity of the cycle $C_n$ is
$$
T(C_n)=\left\{
         \begin{array}{ll}
           \frac{n+3}{n-1}, & \hbox{if $n$ is odd;} \\
           \frac{n+2}{n}, & \hbox{if $n$ is even.}
         \end{array}
       \right.
$$
\end{lemma}

\begin{theorem}
Among all the trees on $n$ vertices,  $G$ has the maximum
tenacity if and only if\\
$(i)$ $G$ is a path when $n$ is even;
\\
$(ii)$ the maximum degree of $G$ is at most 3 and it contains no
nonadjacent vertices of degree 3 when $n$ is odd.
\end{theorem}

\begin{proof}
$(i)$ Suppose that $G$ has the maximum tenacity among all trees on
$n$ vertices and contains a vertex of degree at least 3, say $v_0$.
Let $X_0=\{v_0\}$. Then $X_0$ is a vertex cut of $G$ such that
$G-X_0$ is a forest and $\omega(G-X_0)\geq |X_0|+2$. It follows from
Lemmas 2 and 3 that $T(G)\leq 1<T(P_n)$, a contradiction. This
implies that $G$ is a path. The sufficiency follows immediately.

$(ii)$ Suppose that $G$ has the maximum tenacity among all trees on
$n$ vertices and contains a vertex of degree at least 4, say $v_0$;
or two nonadjacent vertices of degree 3, say $u_0$ and $v_0$. Let
$X_0=\{v_0\}$ or $X_0=\{u_0,v_0\}$. Then $X_0$ is a vertex cut such
that $G-X_0$ is a forest and $\omega(G-X_0)\geq |X_0|+3$. It follows
from Lemmas 2 and 3 that $T(G)<1=T(P_n)$. This completes the proof
of the necessity.

Now let $G$ be a tree on $n$ vertices with maximum degree 3 and
contains no nonadjacent vertices of degree 3. To prove the
sufficiency, we need only show that $G$ has the same tenacity as the
path $P_n$.

\setcounter{case}{0}
\begin{case}
$G$ has only one vertex of degree 3.
\end{case}

Let $u$ be the vertex of degree 3 and $X$ be an arbitrary vertex cut
of $G$. If $u\notin X$, or $u\in X$ and $X$ contains two adjacent
vertices of $G-u$ or a leaf of $G-u$, then we have $\omega(G-X)\leq
|X|+1$, and therefore,
$$
\frac{|X|+\tau(G-X)}{\omega(G-X)}\geq 1.
$$
If $u\in X$ and $X$  contains no adjacent vertices of $G-u$ and
leaves of $G-u$, then we can see that $\omega(G-X)=|X|+2$ and at
least one of the three components of $G-u$ is an even path, say
$P_e$. It is not difficult to see that
$$
\tau(G-X)\geq\tau(P_e-X\cap V(P_e))\geq 2.
$$
Therefore,
$$
\frac{|X|+\tau(G-X)}{\omega(G-X)}\geq
\frac{\omega(G-X)-2+\tau(P_e-X\cap V(P_e))}{\omega(G-X)}\geq 1.
$$
From the above discussion and the definition of tenacity we can see
that $T(G)=T(P_n)$.

\begin{case}
$G$ has exactly two adjacent vertices of degree 3.
\end{case}

Let $u$ and $v$ be the two adjacent vertices of degree 3 and $X$ be
an arbitrary vertex cut of $G$. If $|X\cap \{u,v\}|\leq 1$, then
similar to Case 1, we can prove that
$$
\frac{|X|+\tau(G-X)}{\omega(G-X)}\geq 1.
$$
So now we assume that $|X\cap \{u,v\}|=2$. If $X\setminus \{u,v\}$
contains adjacent vertices of $G$ or a leaf of $G$, then
$\omega(G-X)\leq |X|+1$, and therefore,
$$
\frac{|X|+\tau(G-X)}{\omega(G-X)}\geq 1.
$$
Otherwise,  we have $\omega(G-X)=|X|+2$ and at least one of the
three components of $G-u$ is an even path, say $P_e$. As above, we
have
$$
\tau(G-X)\geq\tau(P_e-X\cap V(P_e))\geq 2.
$$
Therefore,
$$
\frac{|X|+\tau(G-X)}{\omega(G-X)}\geq
\frac{\omega(G-X)-2+\tau(P_e-X\cap V(P_e))}{\omega(G-X)}\geq 1.
$$
From the above discussion and the definition of tenacity we can see
that $T(G)=T(P_n)$.

The proof is complete.
\end{proof}

\begin{cor}
Among all the trees on $n$ vertices, the maximum tenacity is
$$
\left\{
         \begin{array}{ll}
           1, & \hbox{if $n$ is odd;} \\
           \frac{n+2}{n}, & \hbox{if $n$ is even.}
         \end{array}
       \right.
$$
\end{cor}

For unicyclic graphs, we have the following

\begin{theorem}
Among all the unicyclic graphs on $n$ vertices, $G$ has the maximum
tenacity if and only if\\
$(i)$ $G$ is a cycle when $n$ is odd;
\\
$(ii)$ the maximum degree of $G$ is at most 3 and it contains no
nonadjacent vertices of degree 3 when $n$ is even.
\end{theorem}

\begin{proof}
Suppose that $G$ has the maximum tenacity among all the unicyclic
graphs on $n$ vertices.

\begin{claim}
The maximum degree of $G$ is at most 3.
\end{claim}

\begin{proof}
If $G$ contains a vertex of degree at least 4, say $u_0$, let
$X_0=\{u_0\}$ if $u_0\in V(C_G)$ and $X_0=\{u_0,v_0\}$ otherwise,
where $v_0$ is a vertex in $U_G$. Then $X_0$ is a vertex cut of $G$
such that $G-X_0$ is a forest and $\omega(G-X_0)\geq |X_0|+2$. It
follows from Lemmas 2 and 4 that $T(G)\leq 1<T(C_n)$, a
contradiction.
\end{proof}

\begin{claim}
$G$ contains no nonadjacent vertices of degree 3.
\end{claim}

\begin{proof}
If $G$ contains nonadjacent vertices of degree 3, then we can choose
two of them, say $u_0$ and $v_0$, such that at least one of them is
in $U_G$. Set $X_0=\{u_0,v_0\}$. Then $X_0$ is a vertex cut of $G$
such that $G-X_0$ is a forest and $\omega(G-X_0)\geq |X_0|+2$. It
follows from Lemmas 2 and 4 that $T(G)\leq 1<T(C_n)$, a
contradiction.
\end{proof}

Now suppose that $G$ is not a cycle. Then by Claims 1 and 2, there
exists a vertex $u_0\in U_G$ such that $G-u_0$ is the disjoint union
of two paths $P$ and $Q$. Assume that $P=u_1u_2\cdots u_l$ and
$Q=v_1v_2\cdots v_m$. Let
$$
X^*=\left\{\begin{array}{ll}
             \{u_0,u_2,u_4,\ldots,v_{l-2},v_2,v_4,\ldots,v_{m-2}\}, & \mbox{if both $l$ and $m$ are even;} \\
             \{u_0,u_2,u_4,\ldots,v_{l-1},v_2,v_4,\ldots,v_{m-1}\}, & \mbox{if both $l$ and $m$ are odd;} \\
             \{u_0,u_2,u_4,\ldots,v_{l-2},v_2,v_4,\ldots,v_{m-1}\}, & \mbox{if $l$ is even and $m$ is
             odd;}\\
             \{u_0,u_2,u_4,\ldots,v_{l-1},v_2,v_4,\ldots,v_{m-2}\}, & \mbox{if $l$ is odd and $m$ is even.}
           \end{array}
           \right.
$$
Then $X^*$ is a vertex cut of $G$ with
$$
\left\{\begin{array}{ll}
             \mbox{$|X^*|=\frac{l+m}{2}-1$, $\tau(G-X^*)=2$, and $\omega(G-X^*)=\frac{l+m}{2}$,}& \mbox{if both $l$ and $m$ are even;} \\
             \mbox{$|X^*|=\frac{l+m}{2}$, $\tau(G-X^*)=1$, and $\omega(G-X^*)=\frac{l+m}{2}+1$,} & \mbox{if both $l$ and $m$ are odd;} \\
            \mbox{$|X^*|=\frac{l+m-1}{2}$, $\tau(G-X^*)=2$, and $\omega(G-X^*)=\frac{l+m+1}{2}$,} & \mbox{if $l+m$ is odd.}
           \end{array}
           \right.
$$
Thus,
$$
\frac{|X^*|+\tau(G-X^*)}{\omega(G-X^*)}=\left\{\begin{array}{ll}
             \frac{l+m+2}{l+m}=\frac{n+1}{n-1},& \mbox{if both $l$ and $m$ are even;} \\
             1, & \mbox{if both $l$ and $m$ are odd;} \\
             \frac{l+m+3}{l+m+1}=\frac{n+2}{n}, & \mbox{if $l+m$ is odd.}
           \end{array}
           \right.
$$
It follows from the definition of tenacity and Lemma 4 that
$$
T(G)\leq
\frac{|X^*|+\tau(G-X^*)}{\omega(G-X^*)}\left\{\begin{array}{ll}
             <T(C_n)=\frac{n+3}{n-1}, & \mbox{if $n$ is odd;} \\
             = T(C_n)=\frac{n+2}{n}, & \mbox{if $n$ is even.}
           \end{array}
           \right.
$$
This completes the proof of $(i)$ and the necessity of $(ii)$.

Suppose now that $G$ is a unicyclic graph on $n$ vertices with
maximum degree 3 and contains no nonadjacent vertices of degree 3.
To prove the sufficiency of $(ii)$, we need only show that $G$ has
the same tenacity as the cycle $C_n$.

If $G$ has only one vertex of degree 3; or exactly two adjacent
vertices of degree 3, both on $C_G$, then it is easy to see that
$P_n$ is a spanning subgraph of $G$. It follows from Lemmas 1, 3 and
4 that
$$
T(G)\geq T(P_n)=\frac{n+2}{n}=T(C_n).
$$
If $G$ has exactly two adjacent vertices of degree 3, one is on
$C_G$, the other is not; or has exactly three pairwise adjacent
vertices of degree 3, all on $C_G$, then it is easy to see that $G$
contains a spanning subgraph $H$ with maximum degree 3 and exactly
one vertex of degree 3.  It follows from Lemma 1, Theorem 2 and
Lemma 4 that
$$
T(G)\geq T(H)=\frac{n+2}{n}=T(C_n).
$$

The proof is now complete.
\end{proof}

\begin{cor}
Among all the unicyclic graphs on $n$ vertices, the maximum tenacity
is
$$
\left\{
         \begin{array}{ll}
           \frac{n+3}{n-1}, & \hbox{if $n$ is odd;} \\
           \frac{n+2}{n}, & \hbox{if $n$ is even.}
         \end{array}
       \right.
$$
\end{cor}

\section{Extreme values of edge vulnerability parameters}

As we noted in Section 1, besides the vertex vulnerability
parameters, the edge-analogues of some of them have also been
proposed, e.g., edge-toughness \cite{Gusfield} for toughness,
edge-integrity \cite{Barefoot_Entringer_Swart_0} for integrity,
edge-tenacity \cite{Piazzal_Robert_Stueckle} for tenacity. As for
the vertex vulnerability parameters, it would be an interesting
problem to determine the extreme values of the edge vulnerability
parameters of graphs with a given number of vertices and edges.

\end{document}